\documentclass[11pt]{article}
\usepackage{amsfonts}
\usepackage{amsmath}
\usepackage{amsthm}
\usepackage{amssymb}
\usepackage{graphicx}

\numberwithin{equation}{section}
\newtheorem{theorem}{Theorem}[section]
\newtheorem{lemma}{Lemma}[section]

\begin{document}

\title{\Large Numerical analysis of a family of optimal distributed control problems governed by an elliptic variational inequality
\thanks{Supported by Project PIP \# 0534 from CONICET-UA
, Rosario, Argentina, and AFOSR-SOARD Grant FA9550-14-1-0122.}}
\author{Olguin Mariela C.\thanks{Departamento de
Matem\'atica, EFB-FCEIA, Univ. Nacional de Rosario, Avda. Pellegrini 250, S2000BPT Rosario, Argentina.
  E-mail: mcolguin@fceia.unr.edu.ar.} \\
\and
Tarzia Domingo A.\thanks{Corresponding author: Departamento de
Matem\'atica-CONICET, FCE, Univ. Austral, Paraguay 1950, S2000FZF
Rosario, Argentina. Tel: +54-341-5223093, Fax:  +54-341-5223001.
E-mail: DTarzia@austral.edu.ar.}}
\date{}

\maketitle

\begin{abstract} The numerical analysis of a family of distributed mixed optimal control problems governed
by elliptic variational inequalities (with parameter $\alpha >0$) is
obtained through the finite element method when its parameter
$h\rightarrow 0$. We also obtain the limit of the discrete optimal
control and the associated state system solutions when
$\alpha\rightarrow \infty$ (for each $h>0$) and a commutative
diagram for two continuous and two discrete optimal control and its
associated state system solutions is obtained when $h\rightarrow 0$
and $\alpha\rightarrow \infty$. Moreover, the double convergence is
also obtained when $(h, \alpha)\rightarrow(0, \infty)$.
\end{abstract}

\section{Introduction}

Following \cite{BOUKROUCHE12}, we consider a bounded domain $\Omega \subset \mathbb{R}^n $ whose regular boundary
$ \partial\Omega= \Gamma_1 \bigcup \Gamma_2$ consists
of the union of two disjoint portions $\Gamma_1$ and $\Gamma_2$ with $meas(\Gamma_1$ )$ > 0$, and we state, for each $\alpha >0$, the
following free boundary system:

\vspace{-.4cm}
\begin{equation}  \label{probl clasico alfa_1} 
   u\geq 0; \,\, u(-\Delta u-g)=0;  \,\, -\Delta u-g \geq 0 \,\, in \,\,  \Omega;
\end{equation}

\vspace{-.5cm}

\begin{equation}  \label{probl clasico alfa_2} 
   -\frac{\partial u}{\partial n}= \alpha (u-b) \,\, on\,\,  \Gamma_1; \,\, -\frac{\partial u}{\partial n} =q
   \,\, on \,\,  \Gamma_2;
\end{equation}

\noindent where the function $g$ in (\ref{probl clasico alfa_1}) can be considered as the internal energy in
$\Omega$, $\alpha > 0$ is the heat transfer coefficient on $\Gamma_1$, $b > 0$ is the constant environment temperature, and $q$
is the heat flux on $\Gamma_2$. The variational formulation of the above problem is given as ( system $(S_\alpha)$):

\noindent Find  $u= u_{\alpha g} \in  K_+$ such that, $\forall v \in K_+$

\begin{equation} \label{S alfa}  
   a_{\alpha} (u_{\alpha g},v-u_{\alpha g})\geq (g, v-u_{\alpha g})_H - (q,v-u_{\alpha g})_Q + \alpha (b,v-u_{\alpha g})_R,
\end{equation}

\noindent where

$$V=H^1(\Omega), \hspace{.3cm} K_+ =\{v \in V: \, v\geq 0 \,\,in\, \Omega \},$$

\[H= L^2(\Omega), \hspace{.2cm} Q=L^2(\Gamma_2),\hspace{.2cm} \text{and} \hspace{.2cm} \,R=L^2(\Gamma_1),\]

\[(u,v)_H= \int_\Omega u\, v \,\, dx, \hspace{.1cm}  (u,v)_Q = \int_{\Gamma_2} u\,v \,\, ds, \hspace{.1cm} (u,v)_R = \int_{\Gamma_1} u\,v \,\, ds,\]

\[a(u,v)= \int_\Omega \nabla u. \nabla v \,\, dx \]

\noindent and

\begin{equation}  \label{ec 1.4}
   a_{\alpha}(u,v)= a(u,v) + \alpha (u,v)_R.
\end{equation}

\noindent We note that $a_1$, and therefore $a_{\alpha}$, is a
bilinear, continuous, symmetric and coercive form on V
\cite{KINDERHLERERSTAMPACCHIA,TABACMANTARZIA89}, that is to say:
there exists a constant $\lambda_1 >0$ and $0 < \lambda_{\alpha} =
\lambda_1 \min\{1,\alpha\}$ such that

 \begin{equation} \label{ec 1.5}
\,\,a_{\alpha}(v,v) \geq \lambda_{\alpha} \,\Arrowvert v \Arrowvert^2_V  \hspace{.5cm} \forall \,\,v \in V.
\end{equation}

\noindent In \cite{BOUKROUCHE12} the following family of continuous distributed optimal control problem associated
with the system $(S_\alpha)$ was considered for each $\alpha > 0$:

\noindent Problem $(P_\alpha)$: Find the distributed optimal control $g_{op_\alpha} \, \in H$ such that

\begin{equation}  \label{Def Palfa}  
  J_\alpha(g_{op_\alpha})=\min_{g \in H} J_\alpha(g)
\end{equation}

\noindent where the quadratic cost functional $J_\alpha:H \rightarrow \mathbb{R}^+_0$ was defined by:

\begin{equation} \label{ec 1.7}
   J_\alpha(g)= \frac{1}{2} \Arrowvert u_{\alpha g} \Arrowvert^2_H + \frac{M}{2} \Arrowvert g \Arrowvert^2_H
\end{equation}

\noindent with $M > 0$ a given constant and $u_{\alpha g}$ is the corresponding solution of the elliptic
variational inequality (\ref{S alfa}) associated to the control $g \in H$.

Several optimal control problems are governed by elliptic
variational inequalities
(\cite{ADAMS},\cite{AITHADI},\cite{BARBU},\cite{BERGONIOUX97c},\cite{BERGONIOUX},
\cite{DELOSREYES},\cite{DELOSREYESMEYER},\cite{LIONS},\cite{MIGNOT76},\cite{MIGNOT84},\cite{TROLTZSCH})
and there exists an abundant literature about continuous and
numerical analysis of optimal control problems governed by elliptic
variational equalities or inequalities
(\cite{BENBELGACEM03},\cite{CASAS02},\cite{CASAS06},\cite{DECKELNICK07},\cite{Falk74},
\cite{GAMALLO11},\cite{GARIBOLDI03},\cite{HASLINGER},\cite{HINTERMULLER08},
\cite{HINTERMULLER13},\cite{HINZE2},
\cite{HINZE09},\cite{ITO},\cite{KUNISCH12},
\cite{MERMRI},\cite{Meyer13},\cite{TARZIA96},\\
\cite{TARZIA99},\cite{YAN})
and by parabolic variational equalities or inequalities
 (\cite{BOUKROUCHE11}, \cite{MENALDITARZIA07}).

The objective of this work is to make the numerical analysis of the continuous
optimal control problem $(P_\alpha)$ which is governed by the
elliptic variational inequality (\ref{S alfa}) by proving the
convergence of a discrete solution to the solution of the continuous
optimal control problem.

In Section 2, we establish the discrete elliptic variational inequality (\ref{ec 2.3})
which is the discrete formulation of the continuous elliptic variational
inequality (\ref{S alfa}), and we obtain that these discrete problem has unique
solutions for all positive $h$. Moreover, we define a family $(P_{h \alpha})$ of discrete optimal
control problems (\ref{ec 2.8}) and, we obtain several properties for the state
system (\ref{ec 2.3}) and for the discrete cost functional $J_{h \alpha}$ defined in (\ref{ec 2.7}).

In Section 3, on adequate functional spaces, we obtain a result of
global strong convergence when the parameter $h \rightarrow 0$ (for
each $\alpha > 0$) and when $\alpha\rightarrow \infty$ (for each
$h>0$) for the discrete state sytems and for the discrete optimal
optimal problem corresponding to $(P_{\alpha})$. We end this work
proving the double convergence of the discrete optimal solutions of
$(P_{h \alpha})$ when $(h,\alpha)\rightarrow (0,\infty)$ obtaining a
complete commutative diagram among two discrete and two continuous
optimal control problems given en Fig. 1. We generalize recent
results obtained for optimal control problems governed by elliptic
variational equalities given in \cite{TARZIA14,TARZIA15}.

\section{Properties of the discretization of the problem ($P_\alpha$)} 

Let $\Omega\subset \mathbb{R}^n$ be a bounded polygonal domain; $b$ a positive constant and $\tau_h$ a regular triangulation
with Lagrange triangles of type 1, constituted by affine-equivalent finite elements of class $C^0$ over $\Omega$ being
$h$ the parameter of the finite element approximation which goes to zero (\cite{BRENNERSCOTT}, \cite{CIARLET}). We
take $h$ equal to the longest side of the triangles $T \in \tau_h$ and we can approximate the sets $V$ and $K_+$ by:

\[ V_h= \{v_h \in C^0(\overline{\Omega}): v_h /_T \in P_1(T), \,\,\forall \,T \,\in \tau_h\}, \]

\[K_{+h}= \{ v_h \in V_{h} : v_h \geq 0 \,\, in\, \Omega\}\]

\noindent  where $P_1(T)$ is the set of the polynomials of degree less
than or equal to $1$ in the triangle $T$. Let $\Pi_h :C^0(\overline{\Omega}) \rightarrow V_h$
be the corresponding linear interpolation operator and  $c_0 > 0$  a constant
(independent of the parameter $h$) such that, if $1< r \leq 2$ (\cite{BRENNERSCOTT}):

\begin{equation} \label{ec 2.1}  
  \Arrowvert v-\Pi_h(v)\Arrowvert_H \, \leq \, c_0 \, h^r \,\Arrowvert v \Arrowvert_r \,\,\, \forall \, v \in H^r(\Omega),
  \end{equation}

\begin{equation} \label{ec 2.2}  
  \Arrowvert v-\Pi_h(v)\Arrowvert_V \, \leq \, c_0 \, h^{r-1} \,\Arrowvert v \Arrowvert_r \,\,\, \forall \, v \in H^r(\Omega).
\end{equation}

\noindent The discrete formulation $(S_{h\alpha})$ of the continuous system $(S_\alpha)$ is, for each $\alpha>0$ , defined as:
Find $u_{h \alpha g} \in K_{+h}$ such that, for all $v_h \in K_{+h}$

\vspace{-.2cm}
\begin{equation} \label{ec 2.3}  
    a_\alpha(u_{h\alpha g},v_h-u_{h\alpha g})\geq (g, v_h-u_{h\alpha g})_H- (q,v_h-u_{h\alpha g})_ Q + \alpha (b,v_h-u_{h\alpha g})_R.
\end{equation}

\begin{theorem} \label{Teo 2.1}

Let $g\in H$ and $q \in Q $ be, then there exists unique solution of the elliptic variational inequality (\ref{ec 2.3}).
\end{theorem}

\begin{proof}
 It follows from the application of Lax-Milgram Theorem (\cite{KINDERHLERERSTAMPACCHIA}, \cite{LIONS}).
\end{proof}

\begin{lemma} \label{L2.1}

A) Let ${g_n}$ and $g \in H$, and $u_{h \alpha g_n}$ and $u_{h \alpha g}\in K_{+h}$ be the associated solutions of the
system $(S_{h\alpha})$ for each $\alpha > 0$. If $g_n \rightharpoonup g$ in $H$ weak, then we have that:

\begin{enumerate}
 \item [i)] $\exists \,C>0$ (independent of $h$, $\alpha$ and of $n$) such that:
\vspace{-.1cm}
 \begin{equation}\label{ec 2.4}  
     \Arrowvert u_{h \alpha g_n}\Arrowvert_V \leq C;
  \end{equation}

 \item [ii)]$\forall \, h > 0,$
  \begin{equation} \label{ec 2.5} 
      lim_{n\rightarrow \infty}\Arrowvert u_{h \alpha g_n}-u_{h \alpha g} \Arrowvert_V =0 .
  \end{equation}
\end{enumerate}

\noindent B) We have that

\[\Arrowvert u_{h \alpha g_2}-u_{h \alpha g_1} \Arrowvert_V \leq \frac{1}{\lambda_\alpha} \Arrowvert g_2-g_1 \Arrowvert_H \]

\noindent where $u_{h \alpha g_i}$ is the associated solution of the
system $(S_{h \alpha})$ for $g_i, \,i=1,2.$
\end{lemma}

\begin{proof}
 We follow a similar methodology as in (\cite{Falk74}, \cite{OLGUIN}).
\end{proof}

\begin{lemma} \label{L2.2}

Let $ u_{\alpha g} \in K_+ \bigcap H^r(\Omega), (1< r \leq 2)$ and $u_{h \alpha g} \in K_{+h}$ be the solutions of the elliptic variational
 inequalities (\ref{S alfa}) and (\ref{ec 2.3}) respectively for the control $g \in H$, there exists a positive constant $C$ such that
 \begin{equation} \label{ec 2.6}
  \Arrowvert u_{h \alpha g} - u_{\alpha g}\Arrowvert_V \leq C(\alpha) h^{(r-1)/2}.
 \end{equation}
\end{lemma}

\begin{proof}

\noindent  If we consider $v=u_{h \alpha g} \in K_{+h} \subset K_+$ in the elliptic variational
inequality (\ref{S alfa}) and $v_h=\Pi_h(u_\alpha g) \in K_{+h}$ in (\ref{ec 2.3}), and calling $w=\Pi_h(u_{\alpha g})-u_{\alpha g}$, we have that:

\[a_\alpha(u_{h \alpha g}-u_{\alpha g},u_{h \alpha g}-u_{\alpha g}) \leq a_\alpha (u_{h \alpha g},w ) -(g,w)_H +(q,w)_Q -\alpha(b,w)_R\]

\noindent By using the coerciveness of $a_\alpha$, the estimation (\ref{ec 2.2}) and by some mathematical computation, we obtain that:

\[ \Arrowvert u_{h \alpha g}-u_{\alpha g}\Arrowvert_V^2 \, \leq \frac{C}{\lambda_\alpha} \Arrowvert \Pi_h(u_{\alpha g})-u_{\alpha g}\Arrowvert_V \leq \frac{C}{\lambda_\alpha} h^{r-1} \Arrowvert u_{\alpha g}\Arrowvert_r\]

\end{proof}

Now, we consider the continuous optimal control problem which was established in (\ref{Def Palfa}). The associated
discrete cost functional  $J_{h \alpha}: H \rightarrow \mathbb{R}^+_0\, $ is defined by the following expression:

 \begin{equation} \label{ec 2.7}
   \,\,  J_{h \alpha}(g)= \frac{1}{2} \Arrowvert u_{h \alpha g} \Arrowvert^2_H + \frac{M}{2} \Arrowvert g \Arrowvert^2_H
\end{equation}

\noindent where $u_{h \alpha g}$ is the unique solution of the
elliptic variational inequality (\ref{ec 2.3}) for a given control
$g \in H$ and a given parameter $\alpha > 0$. \noindent Then, we
establish the following discrete distributed optimal control problem
$(P_{h\alpha})$: Find $g_{op_{h \alpha}} \,\in H$ such that

\begin{equation} \label{ec 2.8}
  J_{h \alpha}(g_{op_{h \alpha}})=\min_{g \,\in\,H} J_{h \alpha}(g).
\end{equation}

\noindent We remark that the discrete (in the space)
distributed optimal control problem $(P_{h \alpha})$ is still an infinite dimensional optimal control problem since the control space $H$ is not discretized.

\begin{theorem} \label{Teo 2.2}
  For the control $g \in H$, the parameters $\alpha > 0$  and $h > 0$, we have:

  \vspace{.1cm}
  a) \[\lim_ {\Arrowvert g \Arrowvert_H \rightarrow \infty} J_{h \alpha}(g) = \infty.\]

  b) $ J_{h \alpha}(g) \geq \frac{M}{2} \Arrowvert g\Arrowvert^2_H - C \,\Arrowvert g\Arrowvert_H$ \, for some constant $C$ independent of
      $h >0$.

  \vspace{.1cm}

  c) The functional $J_{h \alpha}$ is a lower weakly semi-continuous application in $H$.

  \vspace{.1cm}

  d) For each $h>0$ and $\alpha >0$, there exists a solution of the discrete distributed optimal control problem (\ref{ec 2.8}).

\end{theorem}

\begin{proof}

From the definition of $J_{h \alpha}(g)$ we obtain a) and b).

c) Let $g_n\rightharpoonup g$ in $H$ weak, then by using the equality $\Arrowvert g_n \Arrowvert^2_H = \Arrowvert g_n - g \Arrowvert^2_H - \Arrowvert g \Arrowvert^2_H + 2 (g_n, g)_H$
we obtain that $\Arrowvert g \Arrowvert_H \leq  \liminf_{n\rightarrow \infty} \, \Arrowvert g_n \Arrowvert_H$. Therefore, we have

\[  \liminf_{n\rightarrow \infty} J_{h \alpha}(g_n) \geq \frac{1}{2} \Arrowvert u_{h \alpha g}\Arrowvert^2_H + \frac{M}{2} \Arrowvert g \Arrowvert^2_H = J_{h \alpha }(g).\]

d) It follows from \cite{LIONS}.

 \end{proof}

\begin{lemma} \label{L 2.3}

 If the continuous state system has the regularity $u_{\alpha g} \in H^r(\Omega)$

 \noindent $(1< r \leq 2)$ for $g \in H$ and $\alpha > 0$, then we have the
 following estimation $\forall g \in H$:

  \begin{equation} \label{ec 2.9}
       | J_{h \alpha}(g)- J_\alpha(g)| \leq C(\alpha) h^{\frac{r-1}{2}}
     \end{equation}

 \noindent where $C$ is a positive constant independent of $h > 0$.

\end{lemma}

\begin{proof} By definition of the discrete cost functional $J_{h \alpha}$, we have:

\begin{equation*}
   J_{h \alpha}(g)-J_{\alpha}(g) = \frac{1}{2}\,(\rVert u_{h \alpha g}\lVert^2_H- \rVert u_{\alpha g}\lVert^2_H)= \frac{1}{2} \rVert u_{h \alpha g}-u_{\alpha g}\lVert^2_H \, + \, (u_{\alpha g}, u_{h \alpha g} - u_{\alpha g})_H
  \end{equation*}

\noindent and therefore, if we apply (\ref{ec 2.6}), it results:

 \begin{equation*}
   \lvert J_{h \alpha}(g)-J{\alpha}(g) \rvert \leq \, (\frac{1}{2}\rVert u_{h \alpha g}-u_{ \alpha g}\lVert_H \,+\, \rVert u_{\alpha g}\lVert_H) \, \rVert u_{h \alpha g}-u_{\alpha g}\lVert_H \leq C(\alpha) \, h^{\frac{r-1}{2}},
 \end{equation*}

\noindent and (\ref{ec 2.9}) holds.

\end{proof}

\section{Results of Convergence} 

\subsection{Convergence when $h \rightarrow 0$}

\begin{theorem} \label{Teo 3.1}

 Let $ u_{\alpha g} \in K_+ \bigcap H^r(\Omega), (1< r \leq 2)$ and $u_{h \alpha g} \in K_{+h}$ be the solutions of the elliptic variational
 inequalities (\ref{S alfa}) and (\ref{ec 2.3}) respectively for the control $g \in H$,
 then $u_{h \alpha g} \rightarrow u_{\alpha g}$ in $V$ when $h \rightarrow 0^+$.
\end{theorem}

\begin{proof}

 Similarly to the part $a)$ of the Lemma 2.1, we can show
 that there exist a constant $C > 0$ such that $ \Arrowvert u_{h
 \alpha g} \Arrowvert_V \leq C, \hspace{0.4cm} \forall \,\,h>0.$
 Therefore, we conclude that there exists $\eta_\alpha \in V$ so that $u_{h
 \alpha g} \rightharpoonup \eta_\alpha$  in $V$ (in $H$ strong) as $h
 \rightarrow 0^+$ and $\eta \in K_+$. On the other hand, given $v \in
 K_+$ let be $v_h=\Pi (v) \in K_{+h}$ for each $h$ such that $v_h\rightarrow v$ in $V$ when
  $h$ goes to zero. Now, by considering $v_h \in K_{+h}$ in the discrete elliptic variational inequality (\ref{ec 2.3}) we get:

 \begin{equation} \label{ec 3.1}
    a_\alpha(u_{h \alpha g},v_h- u_{h \alpha g}) \geq (g, v_h-u_{h \alpha g})_H- (q,v_h-u_{h \alpha g})_Q + \alpha (b,v_h-u_{h \alpha g})_R
  \end{equation}

 \noindent and when we pass to the limit as $h\rightarrow 0^+$ in (\ref{ec 3.1}) by using that the bilinear form $a$ is lower
 weak semi-continuous in $V$, we obtain:

 \[a_\alpha(\eta_\alpha, v - \eta_\alpha) \geq (g, v - \eta_\alpha)_H - (q, v-\eta_\alpha)_Q + \alpha (b,v-\eta_\alpha)_R, \hspace{.4cm} \forall \,v\in K_+\]

 \noindent and from the uniqueness of the solution of the discrete elliptic variational inequality (\ref{S alfa}), we obtain that $ \eta=u_{\alpha g}.$

 \noindent Now, we will prove the strong convergence. As consequence of $Lemma\,\, 2.2$, by passing to the limit when $h\rightarrow 0^+$ in the inequality (\ref{ec 2.6}), it results:
\[\lim_{h\rightarrow 0^+} \Arrowvert u_{h \alpha g}-u_{\alpha g}\Arrowvert_V =0.\]
\end{proof}

Henceforth we will consider the following:

\noindent \textbf{Definition} Given $\mu \in [0,1]$ and $g_1, g_2 \in H$, we define:

  \begin{itemize}
   \item [a)] the convex combinations of two data $g_1$ and $g_2$ as
              \begin{equation} \label{ec 3.2}
                g_3(\mu) = \mu \, g_1 + (1-\mu) g_2\,\,\in H,
          \end{equation}

   \item [b)] the convex combination of two discrete solutions

        \begin{equation}  \label{ec 3.3}
                u_{h\alpha 3}(\mu)= \mu \, u_{h \alpha g_1} + (1- \mu) u_{h \alpha g_2} \,\,\,\in K_{+h}
        \end{equation}

   \item [c)] and $u_{h\alpha 4}(\mu)$ as the associated discrete state system which is the solution of the discrete elliptic variational
          inequality (\ref{ec 2.3}) for the control $g_3(\mu)$.

  \end{itemize}

\noindent Following the idea given in \cite{BOUKROUCHE12,OLGUIN} we define two open problems. Given the controls $g_1, g_2 \in H$,

\begin{itemize}
 \item [a)]  \begin{equation} \label{ec 3.4}
       0 \leq u_{h \alpha 4} (\mu) \leq u_{h \alpha 3}( \mu) \hspace{0.1cm} in \hspace{0.1cm} \Omega, \hspace{.5cm} \forall \,\, \mu \in [0, 1], \,\forall h>0,
    \end{equation}

 \item [b)] \begin{equation} \label{ec 3.5}
       \rVert u_{h \alpha 4} (\mu) \lVert_H\, \leq \, \rVert u_{h \alpha 3}( \mu)\lVert_H  \hspace{.5cm} \forall \,\, \mu \in [0, 1], \,\forall h>0.
    \end{equation}
\end{itemize}

\noindent\textbf{Remark 1}: We have that $(\ref{ec 3.4})  \Rightarrow  (\ref{ec 3.5})$.

\noindent\textbf{Remark 2}: If (\ref{ec 3.4}) (or (\ref{ec 3.5})) is true, then the functional $J_{h \alpha}$ is H-elliptic and a strictly convex
application because we have:

\noindent $i)$ \[\rVert g_{3 \mu} \lVert^2_H = \mu \rVert g_1 \lVert^2_H + (1-\mu) \rVert g_2 \lVert^2_H - \mu (1-\mu) \rVert g_2- g_1 \lVert^2_H  \,\,\,\forall g_1, g_2 \in H, \forall \mu \in [0,1]\]

\noindent $ii)$ \[\rVert u_{h \alpha 3}(\mu) \lVert^2_H = \mu \rVert u_{h \alpha 1} \lVert^2_H + (1-\mu) \rVert u_{h \alpha 2} \lVert^2_H - \mu (1-\mu) \rVert u_{h \alpha 2}-u_{h \alpha 1} \lVert^2_H\]
$\forall g_1, g_2 \in H, \forall \mu \in [0,1],\,\, \forall \alpha > 0$.

\noindent Then we get:

  $$\mu J_{h \alpha}(g_1) + (1-\mu) J_{h \alpha}(g_2)-J_{h \alpha}(g_3(\mu))$$

 $$= \frac{\mu(1-\mu)}{2}\lVert u_{h  \alpha g_2}-u_{h  \alpha g_1} \rVert^2_H + \frac{M}{2} \, \mu (1-\mu)\, \lVert g_2-g_1 \rVert^2_H + \frac{1}{2} \big[\lVert u_{h  \alpha 3} \rVert^2_H- \lVert u_{h  \alpha 4} \rVert^2_H \big]$$
 $$\geq \frac{\mu(1-\mu)}{2}\lVert u_{h  \alpha g_2}-u_{h  \alpha g_1} \rVert^2_H + \frac{M}{2} \, \mu (1-\mu)\, \lVert g_2-g_1 \rVert^2_H \geq $$
 $$\frac{M}{2} \, \mu (1-\mu)\, \lVert g_2-g_1 \rVert^2_H > 0 \hspace{0.2cm} \forall \mu \in (0,1),\,\, g_1 \neq g_2 \in H$$

\noindent and therefore, the uniqueness for the discrete optimal control problem $(P_{h \alpha})$, defined in (\ref{ec 2.8}), holds.

\begin{theorem} \label{Teo 3.2}

 Let $u_{\alpha g_{op}} \in K_+$ be the continuous state system associated to the optimal control $g_{op_{\alpha}} \in H$
 which is the solution of the continuous distributed optimal control problem (\ref{Def Palfa}).
 If, for each $h>0$, we choose an discrete optimal control $g_{op_{h \alpha}} \in H$  which is a solution of the discrete distributed optimal control problem (\ref{ec 2.8})
 and its corresponding discrete state system $u_{h \alpha \,g_{op_{h \alpha}}} \in K_{+h}$, we obtain that:

\begin{equation} \label{ec 3.6}
  u_{h \alpha \,g_{op_{h \alpha}}} \rightarrow u_{\alpha g_{op_\alpha}} \,\,\, in \,\,\, V \,\,\, strong \,\,\, when \,\,\, h\rightarrow 0^+,
\end{equation}

\noindent and

\begin{equation} \label{ec 3.7}
  g_{op_{h \alpha}} \rightarrow g_{op_\alpha}  \,\,\, in \,\,\,  H \,\,\, strong \,\,\, when \,\,\, h\rightarrow 0^+.
\end{equation}

\end{theorem}

\begin{proof}
 Now, we consider a fixed value of the heat transfer coefficient $\alpha > 0$. Let be $h >0$ and $g_{op_{h \alpha}}$ a solution of
 (\ref{ec 2.8}) and $u_{h\alpha g_{op_{h \alpha}}} $ its associated
discrete optimal state system which is the solution of the problem defined in (\ref{ec 2.3}) for each $h >0$.
From (\ref{ec 2.7}) and (\ref{ec 2.8}), we have that for all $g \in H$

 \begin{equation*}
  J_{h \alpha}(g_{op_{h \alpha}})= \frac{1}{2} \lVert u_{h \alpha g_{op_{h \alpha}}} \rVert^2_H + \frac{M}{2} \lVert g_{op_{h \alpha}} \rVert^2_H\, \leq \, \frac{1}{2} \lVert u_{h \alpha g} \rVert^2_H + \frac{M}{2} \lVert g \rVert^2_H.
 \end{equation*}

\noindent Then, if we consider $g=0$  and $u_{h \alpha 0}$ his corresponding associated state system, it results that:

\begin{equation*}
  J_{h \alpha}(g_{op_{h \alpha}})= \frac{1}{2} \lVert u_{h \alpha g_{op_{h \alpha}}} \rVert^2_H + \frac{M}{2} \lVert g_{op_{h \alpha}} \rVert^2_H\, \leq \, \frac{1}{2} \lVert u_{h \alpha 0} \rVert^2_H.
 \end{equation*}

\noindent Since $\lVert u_{h \alpha 0} \rVert_H \leq C \hspace{0.4cm} \forall \hspace{0.2cm} h$, then we can obtain:

  \begin{equation} \label{ec 3.8}
        \lVert u_{h \alpha g_{op_{h \alpha}}} \rVert_H \leq C \hspace{0.4cm} \forall \hspace{0.2cm} h
     \end{equation}

\noindent and

     \begin{equation}  \label{ec 3.9}
        \lVert g_{op_{h \alpha}} \rVert_H \, \leq \, \frac{1}{\sqrt{M}} \, \lVert u_{h \alpha 0} \rVert_H \leq \frac{1}{\sqrt{M}}\, C  \hspace{0.4cm} \forall \hspace{0.2cm} h.
     \end{equation}

\vspace{.2cm}

\noindent If we consider $v_h=b \in K_{+h} $ in the inequality (\ref{ec 2.3}) for $g_{op_{h \alpha}}$ we obtain, because the coerciveness of the application $a_\alpha$:

\begin{equation} \label{ec 3.10}
  \lVert u_{h \alpha g_{op_{h \alpha}}} \rVert_V \leq C
\end{equation}

\noindent where the constant $C$ is independent of the parameter $h$ y $\alpha > 0$. Now we can say that there exist $\eta_\alpha \,\, \in \,\,V$ and $f_\alpha \,\, \in \,\,H$
such that $u_{h \alpha g_{op_{h \alpha}}}\rightharpoonup \eta_\alpha$ in $V$ weak
(in $H$ strong), and $g_{op_{h \alpha}}\rightharpoonup f_\alpha$ in $H$  weak when $h\rightarrow 0^+$. Moreover, $\eta_\alpha \in K_+$.

Let given $v \in K_+$, by taking $v_h=\Pi(v) \in K_{+h}$ we know that
$v_h\rightarrow v$ in $V$ strong when $h\rightarrow 0^+$. Then, if
we consider the variational elliptic inequality (\ref{ec 2.3}) for
$g= g_{op_{h \alpha}}$ we have, taking into account that the
application $a_\alpha$ is a lower weak semi-continuous application
in $V$, that:

\begin{equation*}
 a_\alpha (\eta_\alpha, v-\eta_\alpha) \geq (f_\alpha, v-\eta_\alpha) - (q, v-\eta_\alpha)_Q + \alpha (b,v-\eta_\alpha)_R, \hspace{1cm} \forall \,\,v \in K_+
\end{equation*}

\noindent and by the uniqueness of the solution of the problem given by the elliptic variational inequality
(\ref{S alfa}), we deduce that $\eta_\alpha = u_{\alpha \,f_\alpha}.$

\noindent By using that the functional cost $J_\alpha$ is
semi-continuous in $H$ weak (see \cite{BOUKROUCHE12}) and Theorem
$3.1$, it results that $f= u_{\alpha {g_{op_\alpha}}}$ and $\eta_\alpha =
u_{g_{op_\alpha}}$.

\noindent Now, we consider $v=u_{h \alpha g_{op_{h \alpha}}} \in K_{+h}\subset K_+$ in the system $(S_\alpha)$ with control
$g_{op_\alpha}$, and $v_h= \Pi_h(u_{\alpha g_{op_\alpha}})$ in the discrete system $(S_{h\alpha})$ for the control
$g_{op_{h\alpha}}$ and define $w_h= u_{h \alpha g_{op_{h \alpha}}}-u_{\alpha g_{op_\alpha}}$.
\noindent After some mathematical work, we obtain that:
\[a_\alpha(w_h,w_h) \leq -a_\alpha(u_{h \alpha g_{op_{h \alpha}}}, \Pi_h(u_{\alpha g_{op_\alpha}})-u_{\alpha {g_{op \alpha}}}) + (q,\Pi_h(u_{\alpha g_{op_\alpha}})-u_{\alpha{g_{op \alpha}}})_Q\]

\[-\alpha (b, \Pi_h(u_{\alpha g_{op_\alpha}})-u_{\alpha{g_{op \alpha}}})_R + (g_{op_{h \alpha}}, \Pi_h(u_{\alpha g_{op_\alpha}})- u_{h \alpha g_{op_{h \alpha}}})_H \]
\[- (g_{op_\alpha},w_h)_H.\]

\noindent From the coerciveness of the application $a_\alpha$, and $ u_{h \alpha g_{op_{h \alpha}}} \rightarrow u_{\alpha g_{op_\alpha}} \,in \,H$  and
$\Pi_h(u_{\alpha g_{op_\alpha}}) \rightarrow u_{\alpha{g_{op \alpha}}} \,in \,H$, we obtain that $\lVert w_h \rVert_V \rightarrow 0$ if
$h \rightarrow 0$ and then (\ref{ec 3.6}) it holds. Its easy to see that (\ref{ec 3.7}) holds too.

\end{proof}

\subsection{Convergence when $\alpha \rightarrow \infty$}

Now, under the same hypothesis in \S 1, we consider the following free boundary system \cite{BOUKROUCHE12}:

\begin{equation}  \label{ec 3.11}
  \,\,\,\, u\geq 0;\,\, u(-\Delta u-g)=0; \,\, -\Delta u-g \geq 0 \,\, in\,\, \Omega;
\end{equation}

\begin{equation}  \label{ec 3.12}
   u=b \,\,\, on\,\, \Gamma_1; \,\,\, -\frac{\partial u}{\partial n} =q \,\,\, on \,\,\,  \Gamma_2;
\end{equation}

\noindent where the function $g$ in (\ref{ec 3.11}) can be considered as the internal energy in $\Omega$, $b$
is the positive constant temperature on $\Gamma_1$ and $q$ is the heat flux on $\Gamma_2$. The variational formulation of
the above problem is given as $(S)$: Find  $ u_g \in  K$ such that

\begin{equation} \label{ec 3.13}
   a(u,v-u_g)\geq (g, v-u_g)_H - (q,v-u_g)_Q,\,\,\, \forall \, v \in K
\end{equation}

\noindent where

\begin{equation*}
 K=\{v \in V: \, v\geq 0 \,\,in\,\, \Omega,\,v{/ \Gamma_1} = b\}.
\end{equation*}

\noindent In \cite{BOUKROUCHE12}, the following continuous distributed optimal control problem $(P)$ associated
with the elliptic variational inequality (\ref{ec 3.13}) was considered: Find the
continuous distributed optimal control $g_{op} \, \in H$ such that

\begin{equation}  \label{ec 3.14}
  J(g_{op})=\min_{g \in H} J(g)
\end{equation}

\noindent where the quadratic cost functional $J:H \rightarrow \mathbb{R}^+_0$ is defined by:

\begin{equation} \label{ec 3.15}
   J(g)= \frac{1}{2} \Arrowvert u_g \Arrowvert^2_H + \frac{M}{2} \Arrowvert g \Arrowvert^2_H
\end{equation}

\noindent with $M > 0$ a given constant and $u_g$ is the corresponding solution of the elliptic
variational inequality (\ref{ec 3.13}) associated to the control $g \in H$.

\noindent Therefore, as in \S 2, we define the discrete variational inequality formulation $(S_h)$ of the
system $(S)$ as follows: Find $u_{hg} \in K_h$ such that

\vspace{-.3cm}

\begin{equation}\label{ec 3.16}
  a(u_{hg},v_h-u_{hg})\geq (g, v_h-u_{hg})_H - (q, v_h-u_{hg})_Q \,\,\, \forall v_h \in K_h.
\end{equation}

\noindent where \[ K_h= \{ v_h \in V_h: v_h \geq 0 \,\, in \,\, \Omega, v_h / \Gamma_1 = b\}\]

The corresponding discrete distributed optimal control problem ($P_h$) of the continuous distributed optimal control problem ($P$)
is defined as: Find the discrete distributed optimal control $g_{op_h} \, \in H$ such that

\begin{equation}  \label{ec 3.17}
 J_h(g_{op_h})=\min_{g \in H} J_h(g)= \frac{1}{2} \Arrowvert u_{h g} \Arrowvert^2_H + \frac{M}{2} \Arrowvert g \Arrowvert^2_H,
\end{equation}

\noindent where $u_{h g}$ is the solution of the elliptic variational inequality (\ref{ec 3.16}).

\begin{theorem} \label{Teo 3.3}

\vspace{.1cm}

\noindent $i)$ Let $g\in H$, and $q \in Q $ be, then there exists unique solution of elliptic variational inequality (\ref{ec 3.16}).

\vspace{.2cm}
\noindent $ii)$ There exists a solution of the discrete optimal control problem (\ref{ec 3.17})

\end{theorem}

\begin{proof}

 $i)$ It follows from the application of Lax-Milgram Theorem \cite{KINDERHLERERSTAMPACCHIA}, \cite{LIONS}.

 $ii)$ It follows from \cite{OLGUIN}.
\end{proof}

\begin{theorem} \label{Teo 3.4}
Let $g\in H$, $q \in Q $ and $h > 0$ be, then we have

\[lim_{\alpha \rightarrow \infty} \Arrowvert u_{h \alpha g}-u_{h g}\Arrowvert_V =0 .\]
\end{theorem}

\begin{proof}
Without loss of generality, we consider $\alpha > 1$ and we define $w= u_{h \alpha g} - u_{h g} \in V$.
By definition of $a_\alpha$, we have:
\[a_\alpha(w,w)-a_1(w,w)=(\alpha-1) \lVert w \rVert_R^2.\]

\noindent After mathematical work, we obtain that:

\begin{equation} \label{ec 3.18}
  a_1(w,w) \leq a_1(w,w)+(\alpha-1) \lVert w \rVert_R^2 \leq (g,w)_H - (q,w)_Q - a(u_{hg},w)
\end{equation}

\noindent and by coerciveness of $a_1$ it results that:

\[\lVert u_{h \alpha g} - u_{h g} \rVert_R^2 \leq \frac{C}{\alpha - 1}\]

\noindent and $u_{h \alpha g} \rightarrow u_{h g}\,\,$in $\,\,\Gamma_1,$ when $\alpha \rightarrow \infty$.

\noindent Moreover, as a consequence  of (\ref{ec 3.18}), we obtain that $\lVert u_{h \alpha g}\rVert_V \leq C$
(C constant independent of $\alpha$ and $h$). Then, there exist $\eta \in V$ such that

\[ u_{h \alpha g} \rightharpoonup \eta \,\, \text{in} \, V \,\, (\text{in $H$ strong}).\]

\noindent Then, the strong convergence in $V$ is obtained similarly to the one in Theorem $3.1$
\end{proof}

\begin{theorem} \label{Teo 3.5}

If, for each $h >0$ we choose $g_{op_{h \alpha}} \in H$ a solution of the optimal control problem $(P_{h \alpha})$
and consider its respective discrete state system $u_{h \alpha \,g_{op_{h \alpha}}} \in K_{+h}$ the solution of (\ref{ec 2.3}),
 we obtain that:

 \begin{equation} \label{ec 3.19}
    u_{h \alpha \,g_{op_{h \alpha}}} \rightarrow u_{h f_h} \,\,\, in \,\,\, V \,\,\, when \,\,\, \alpha \rightarrow \infty,
 \end{equation}

\noindent and

\begin{equation} \label{ec 3.20}
 g_{op_{h \alpha}} \rightarrow f_h  \,\,\, in \,\,\,  H \,\,\,\, when \,\,\, \alpha \rightarrow \infty.
\end{equation}

\noindent where $f_h \in H$ is a solution of the discrete optimal control problem  $(P_h)$ and $u_{h f_h}$ is its corresponding discrete state system
solution of the variational inequality (\ref{ec 3.16}).

\end{theorem}

\begin{proof}

 As in Theorem 3.2, the inequalities (\ref{ec 3.8}) and (\ref{ec 3.9})hold. Now, considering
 $v_h= b$ in (\ref{ec 2.3})(and we take $\alpha >1$ without loss of generality) for the control $g_{op_{h \alpha}}$
 and $w_h= b- u_{h \alpha \,g_{op_{h \alpha}}}$,
 we obtain:

 \[a_\alpha(u_{h \alpha \,g_{op_{h \alpha}}},w_h) \geq (g_{op_{h \alpha}}, w_h)_H - (q,w_h)_Q + \alpha (b,w_h)_R \]

 \noindent that is to say:

\begin{equation} \label{ec 3.21}
   a_1(- w_h,w_h) + a_1(b,w_h) \geq (g_{op_{h \alpha}},w_h)_H -(q,w_h)_Q + (\alpha-1) \lVert w_h\rVert_R.
\end{equation}

 \noindent By the coerciveness of the application $a_1$, it results that:

\begin{equation}  \label{ec 3.22}
   \lVert u_{h \alpha \,g_{op_{h \alpha}}}\rVert_V \leq C \,\,\, \forall \, \alpha > 0.
\end{equation}

 \noindent Moreover,

 \begin{equation}\label{ec 3.23}
    \lVert u_{h \alpha \,g_{op_{h \alpha}}}\rVert_R \leq \frac{C}{\alpha-1} \,\,\, \forall \, \alpha > 0.
 \end{equation}

\noindent Then, there exists $f_h \in H$ and $\eta_h \in V$ (we can see that $\eta_h \in K_h$) such that

\begin{equation} \label{ec 3.24}
 g_{op_{h \alpha}} \rightharpoonup f_h \,\,\text{in}\,\, H
\end{equation}

\noindent and

\begin{equation}\label{ec 3.25}
 u_{h \alpha \,g_{op_{h \alpha}}} \rightharpoonup \eta_h \,\,\text{in} \,\,V\,\, \text{(in $H$ strong)}
\end{equation}

\noindent Let be $v_h \in K_h \subset K_{+h}$ and given $w_h=v_h - u_{h \alpha \,g_{op_{h \alpha}}}$, we have:

\[a_\alpha(u_{h \alpha \,g_{op_{h \alpha}}},w_h) \geq (g_{op_{h \alpha}}, w_h)_H - (q,w_h)_Q + \alpha(b,w_h)_R\]

\[a(u_{h \alpha \,g_{op_{h \alpha}}}, w_h) \geq (g_{op_{h\alpha}}, w_h)_H - (q, w_h)_Q + \alpha (b-u_{h \alpha \,g_{op_{h \alpha}}}, w_h)_R\]

\noindent and because of (\ref{ec 3.24}), (\ref{ec 3.25}) and similar arguments given in Theorem 3.2,
and the fact that the application $a$ is
semi-continuous in $V$ weak, we obtain that $\eta_h$ is a solution of (\ref{ec 3.16}) for the control $f_h$.
Then (by item (i) in Theorem 3.3), $\eta_h=u_{hf_h}.$

If we consider $v_h = u_{hf_h} \in K_h \subset K_{+h}$ in (\ref{ec 2.3}) for the control $g_{op_{h \alpha}} \in H$ and
$w_h = u_{hf_h} -u_{h \alpha \,g_{op_{h \alpha}}}$, then:

\[a_\alpha(u_{h \alpha \,g_{op_{h \alpha}}},w_h) \geq (g_{op_{h \alpha}}, w_h)_H - (q,w_h)_Q + \alpha (b,w_h)_R \]

\[a_1( w_h,w_h) \leq a_1(u_{hf_h},w_h) - (g_{op_{h \alpha}},w_h)_H + (q,w_h)_Q - \alpha (b,w_h)_R \,+ \]
\[(\alpha - 1) (u_{h \alpha g_{op_{h \alpha}}},w_h)_R.\]

\noindent Again, as consequence of the coerciveness of the application $a_1$ and by (\ref{ec 3.24}), (\ref{ec 3.25}), it results
(\ref{ec 3.19}).

Now we see that $f_h$ is a solution of (\ref{ec 2.8}): because of Theorem 2.2 (c),and by the definition of optimum:

\[J_h(f_h)\leq \lim_ {\alpha \rightarrow \infty} J_{h \alpha}(g_{op_{h \alpha}}) \leq \lim_ {\alpha \rightarrow \infty} J_{h \alpha}(g) \,\,\ \forall \, g\in H \]

\noindent and by Theorem 3.4 we conclude that
\[J_h(f_h)\leq J_h (g)  \,\,\ \forall \, g\in H. \]

\noindent Finally, we see that:

\[J_h(f_h)\leq \lim_ {\alpha \rightarrow \infty} J_{h \alpha}(g_{op_{h \alpha}}) \leq J_h(g) \,\,\ \forall \, g\in H \]

\noindent then, if we consider $g=f_h$: \[\lim_ {\alpha \rightarrow \infty} J_{h \alpha}(g_{op_{h \alpha}}) = J_h(f_h)\]

\noindent and, because (\ref{ec 3.19}),

\begin{equation} \label{ec 3.26}
 \lim_ {\alpha \rightarrow \infty} \lVert g_{op_{h \alpha}} \rVert_H = \lVert f \rVert_H.
\end{equation}

Then, by using (\ref{ec 3.24}), (\ref{ec 3.25}), we obtain (\ref{ec 3.20}).

\end{proof}

\noindent Now, following the idea given in \cite{TARZIA15} we have this final theorem:

\subsection{Double convergence when $(h, \alpha) \rightarrow (0^+,\infty)$}
\begin{theorem} \label{Teo 4.4} 

If, for each $h >0$ we choose $g_{op_{h \alpha}} \in H$ a solution of the optimal control problem $(P_{h \alpha})$
and we consider its respective discrete state system $u_{h \alpha \,g_{op_{h \alpha}}} \in K_{+h}$, which is the unique solution of (\ref{ec 2.3}),
 we obtain that:

  \begin{equation} \label{ec 3.27}
      u_{h \alpha \,g_{op_{h \alpha}}} \rightarrow u_{g_{op}} \,\,\, in \,\,\, V \, \,\, when \,\,\, (h,\alpha) \rightarrow (0^+,\infty),
   \end{equation}

   \noindent and

\begin{equation} \label{ec 3.28}
 g_{op_{h \alpha}} \rightarrow g_{op}  \,\,\, in \,\,\,  H \,\,\,\,\, when \,\,\, (h,\alpha) \rightarrow (0^+, \infty).
\end{equation}

\noindent where $ g_{op}\in H$ is the solution of the optimal control problem  $(P)$ and $u_{g_{op}}$ is its corresponding state system
solution of the variational inequality (\ref{ec 3.13}).

\end{theorem}

\begin{proof}
 As in Theorem $3.2$, we have (\ref{ec 3.8}), (\ref{ec 3.9}) and (3.10)
and, in consequence, there exist $u^* \, \in \,V$ with $u^*/\Gamma_1=b$ and $g^* \in H$ such that:

\begin{equation}  \label{ec 3.29}
  u_{h \alpha g_{op_{h \alpha}}} \rightharpoonup u^* \hspace{1cm}(strong \quad in \quad H)
\end{equation}

\noindent and

\begin{equation}  \label{ec 3.30}
  g_{op_{h \alpha}} \rightharpoonup g^*
\end{equation}

\noindent when $(h,\alpha) \rightarrow (0^+,\infty)$ in both cases. Let be $v \in K$ such that $v/\Gamma_1=b$.
We consider $v_h=\Pi_h(v) \in \,K_{+h}$ in the state system (\ref{ec 2.3}) and we define
$w_h=v_h-u_{h \alpha g_{op_{h \alpha}}}$. Then we obtain

\[a(u_{h \alpha g_{op_{h \alpha}}}, w_h) \geq (g_{op_{h \alpha}}, w_h)_H - (q, w_h)_Q.\]

\noindent Because the application $a$ is semi-continuous weak in $V$ and $w_h \longrightarrow v-u^*$ in $H$ when $(h,\alpha) \rightarrow (0, \infty)$,
it results that $u^*$ is solution of (\ref{ec 3.13}). But this problem has unique solution, then we conclude that $u^*= u_{g^*}$
Moreover, we have that:
\[a_\alpha(u_{h\alpha g_{op_{h \alpha}}}-u_{g^*},u_{h\alpha g_{op_{h \alpha}}}-u_{g^*}) \leq (g_{op_{h \alpha}}, u_{h\alpha g_{op_{h \alpha}}} - \Pi(u_{g^*}))_H\]
\[+ (q, u_{h\alpha g_{op_{h \alpha}}} - \Pi(u_{g^*}))_Q + \alpha (b, u_{h\alpha g_{op_{h \alpha}}} - \Pi(u_{g^*}))_R - a_\alpha(u_{g^*},u_{h\alpha g_{op_{h \alpha}}}- \Pi(u_{g^*}) )\]
\[ + a_\alpha(u_{h\alpha g_{op_{h \alpha}}}, \Pi(u_{g^*})-u_{g^*}) - a_\alpha(u_{g^*}, \Pi(u_{g^*})-u_{g^*}).\]

Beacause the coerciveness of the application $a_\alpha $ in $V$ and by (\ref{ec 3.29}) and (\ref{ec 3.30}), we obtain (\ref{ec 3.27}) when
$(h,\alpha) \rightarrow (0^+,\infty).$

As the functional $J_{h\alpha}$ is lower weakly semi-continuous in $H$ (Theorem $2.2$) and (\ref{ec 3.30}) we obtain that $g_{op_{h \alpha}} \rightharpoonup g_{op}$.

\noindent  We also have that $ lim_{(h,\alpha)\rightarrow (0,
\infty)} J_{h \alpha} (g_{op_{h \alpha}}) = J(g_{op})$, and then

\noindent $ lim_{(h,\alpha)\rightarrow (0, \infty)} \lVert g_{op_{h
\alpha}} \rVert_H = \lVert g_{op}\rVert_H$, and by (\ref{ec 3.30}),
(\ref{ec 3.28}) the thesis holds.
\end{proof}

\newpage

\section{Conclusion}

In conclusion, by using the previous results given in \cite{BOUKROUCHE12}, and \cite{OLGUIN} we obtain the following commutative diagram among the two
continuous optimal control problems $(P)$ and $(P_\alpha)$, and two discrete optimal control problems $(P_h)$ and $(P_{h\alpha})$ when $h\rightarrow 0$,
 $\alpha\rightarrow \infty$ and $(h,\alpha) \rightarrow (0^+, \infty)$, which can be summarized by the following figure (Fig. 1):

 \begin{figure}[!ht]
 \begin{center}
  \includegraphics[width=16.5cm]{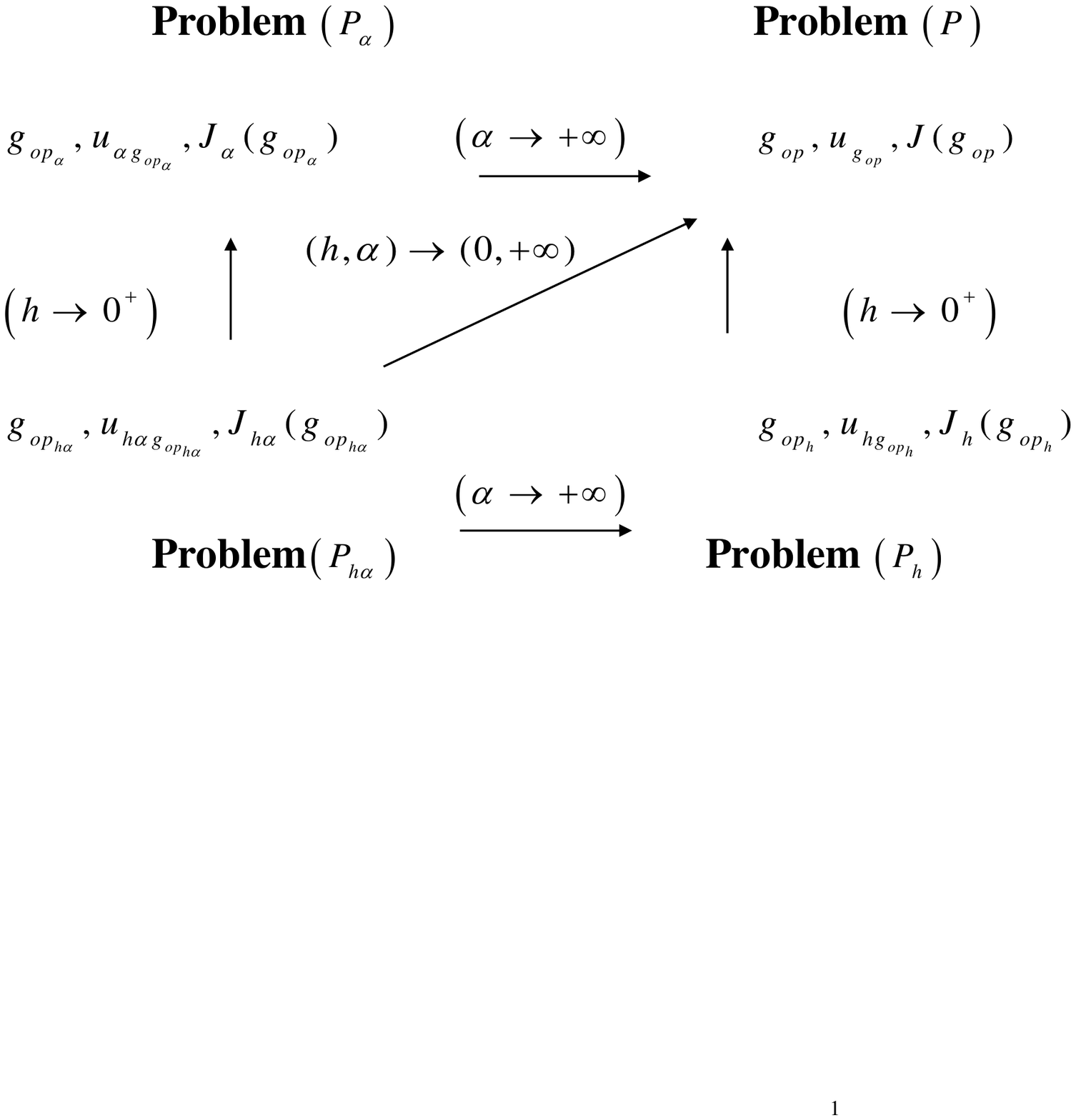}
 \end{center}
 \end{figure}
 \vspace{-5cm}

\begin{center}
 \emph{Fig. 1: Complete diagram for two continuous and two discrete optimal control and
the associated state system solutions}
\end{center}

\end{document}